\documentclass{amsart}

\usepackage{amscd,amsxtra,amsthm}
\usepackage[all]{xy}
\usepackage{etex}
\usepackage{pictex}
\usepackage{graphicx}
\usepackage{mathtools}
\usepackage{float}

\usepackage{multirow}
\DeclarePairedDelimiter{\ceil}{\lceil}{\rceil}

\usepackage{comment}

\newtheorem{theorem}{Theorem}

\theoremstyle{definition}
\newtheorem{definition}[theorem]{Definition}

\theoremstyle{corollary}
\newtheorem{corollary}[theorem]{Corollary}

\theoremstyle{conjecture}
\newtheorem{conjecture}[theorem]{Conjecture}

\theoremstyle{remark}

\theoremstyle{proposition}

\def\Z{\mathbb{Z}}

\def\N{\mathbb{N}}

\begin{document}
\parskip10pt
\parindent15pt
\baselineskip15pt    

\title{Trees and $n$-Good Hypergraphs}

\author[M. Budden]{Mark Budden}
\address{Department of Mathematics and Computer Science \\
Western Carolina University \\
Cullowhee, NC 28723 USA}
\email{mrbudden@email.wcu.edu}

\author[A. Penland]{Andrew Penland}
\address{Department of Mathematics and Computer Science \\
Western Carolina University \\
Cullowhee, NC 28723 USA}
\email{adpenland@email.wcu.edu}

\subjclass[2010]{Primary 05C05, 05C65; Secondary 05C55, 05D10}
\keywords{Ramsey numbers, hypertrees, loose paths}

\begin{abstract}
Trees fill many extremal roles in graph theory, being minimally connected and serving a critical role in the definition of $n$-good graphs.  In this article, we consider the generalization of trees to the setting of $r$-uniform hypergraphs and how one may extend the notion of $n$-good graphs to this setting.  We prove numerous bounds for $r$-uniform hypergraph Ramsey numbers involving trees and complete hypergraphs and show that in the $3$-uniform case, all trees are $n$-good when $n$ is odd or $n$ falls into specified even cases. \end{abstract}

\maketitle

\section{Introduction}

In graph theory, trees play the important role of being minimally connected.  The removal of any edge results in a disconnected graph.  So, it is no surprise that trees serve as optimal graphs with regard to certain extremal properties, especially in Ramsey theory.  Here, one defines the Ramsey number $R(G_1, G_2)$ to be the minimal natural number $p$ such that every red/blue coloring of the edges in the complete graph $K_p$ on $p$ vertices contains a red subgraph isomorphic to $G_1$ or a blue subgraph isomorphic to $G_2$.  

In 1972, Chv\'atal and Harary \cite{CH} proved that for all graphs $G_1$ and $G_2$, \begin{equation} R(G_1, G_2)\ge (c(G_1)-1)(\chi (G_2)-1)+1,\label{CHineq} \end{equation} where $c(G_1)$ is the order of a maximal connected component of $G_1$ and $\chi (G_2)$ is the chromatic number for $G_2$.  Burr \cite{B} was able to strengthen this result slightly by proving that \begin{equation} R(G_1, G_2)\ge (c(G_1)-1)(\chi (G_2)-1)+t(G_2),\label{Burrineq} \end{equation} where $t(G_2)$ is the minimum number of vertices in any color class of any (vertex) coloring of $G_2$ having $\chi (G_2)$ colors.  If we consider the special case in which $G_2=K_n$, we find that (\ref{CHineq}) and (\ref{Burrineq}) agree, and we have  $$R(G_1, K_n )\ge (c(G_1)-1)(n-1)+1.$$  With regard to this inequality, Chv\'atal \cite{C} was able to prove that trees are optimal: $$R(T_m, K_n )=(m-1)(n-1)+1,$$ where $T_m$ is any tree on $m$ vertices.  In particular, it follows that $$R(T_m, K_n)=R(T_m', K_n)$$ for any two trees $T_m$ and $T_m'$ having $m$ vertices and $$R(G, K_n )\ge R(T_m, K_n ),$$ for all graphs $G$ with $c(G)=m$.

The optimal role trees possess in this Ramsey theoretic setting led Burr and Erd\H{o}s \cite{BE} to refer to a connected graph $G$ of order $m$ as being  {\it $n$-good} if it satisfies $$R(G, K_n )=R(T_m, K_n )=(m-1)(n-1)+1.$$  That is, $G$ is $n$-good if the Ramsey number $R(G,K_n)$ is equal to the lower bound given by Chv\'atal and Harary \cite{CH} and Burr \cite{B} (which happen to agree in this case).  Our goal in the present paper is to extend the definition of $n$-good to the setting of $r$-uniform hypergraphs and to consider how one can generalize and adapt the results of \cite{B} and \cite{BE} to the hypergraph setting.

In Section 2, we establish the definitions and notations to be used through the remainder of the paper.  As $r$-uniform trees will serve an important role in our definitions, we also prove a couple of important results concerning such hypergraphs that are analogues of known results in the graph setting.  In Section 3, we prove a generalization of a Ramsey number lower bound that is due to Burr \cite{B} and use this result to define the concept of an $n$-good hypergraph.  Several other Ramsey number inequalities are also given in the section, some which generalize results from the theory of graphs. 

In Section 4, we ask the question ``are $r$-uniform trees $n$-good?''  A previous result due to Loh \cite{L} provides a partial answer and we set the stage for addressing this question in general.  Although we are unable to provide a complete answer, we are able to make significant progress on the determination of which $3$-uniform trees are $n$-good in Section 5.  In this section, we show that infinitely many $3$-uniform trees are $n$-good without finding a single counterexample. We also consider examples of $3$-uniform loose cycles and we conclude with some conjectures regarding the Ramsey numbers for $r$-uniform trees versus complete hypergraphs.

\noindent {\bf Acknowledgement:}  The authors would like to thank John Asplund for carefully reading a preliminary draft of this paper and for making several valuable comments about its content.

\section{Background on Hypergraphs and Trees}

Recall that an {\it $r$-uniform hypergraph} $H=(V(H), E(H))$ consists of a nonempty set of vertices $V(H)$ and a set of $r$-uniform hyperedges $E(H)$ in which each hyperedge is an unordered $r$-tuple of distinct vertices in $V(H)$ (of course, $r=2$ corresponds to graphs).  Also, if $|V(H)|<r$, then $E(H)$ is necessarily empty.   
The \textit{complete $r$-uniform hypergraph} $K_{n}^{(r)}$ consists of the vertex set $V = \{1,2,\ldots,n\}$, with all $r$-element subsets of $V$ as hyperedges. 
If $H_1$ and $H_2$ are $r$-uniform hypergraphs, the \textit{Ramsey number} $R(H_1, H_2 ;r)$ is the smallest natural number $p$ such that any red-blue coloring of $K_{p}^{(r)}$ contains either a red copy of $H_1$ or a blue copy of $H_2$. 

As with graphs, there are various types of colorings of hypergraphs. A \textit{weak proper vertex coloring} of a hypergraph $H$ is a function $\chi$ from $V(H)$ to a \textit{color set} $C$ such that there is no hyperedge with all vertices taking the same value in $C$. A \textit{color class} is a set of vertices which all share the same color in a given coloring. The size of the smallest color set such that there exists a proper vertex coloring of $H$ is the \textit{weak chromatic number of H}, denoted $\chi_w(H)$. We write $t(H)$ for the minimum size of a color class in any coloring of $H$ with $\chi_w(H)$ colors. 

In the graph setting, several equivalent definitions of trees are used, and we begin our analysis by considering equivalent definitions in the hypergraph setting.  A few special terms and properties need to be defined.  
For $r$-uniform hypergraphs with $r>2$, paths and cycles have many more degrees of freedom than their graphical counterparts.  In the broadest sense, a {\it Berge path} is a sequence of $k$ distinct vertices $v_1, v_2, \dots , v_k$ and $k-1$ distinct $r$-edges $e_1, e_2, \dots , e_{k-1}$ such that for all $i\in \{ 1, 2, \dots , k-1\}$, $v_i, v_{i+1} \in e_i$.  
\begin{figure}[h!]
\centerline{{\includegraphics[width=0.7\textwidth]{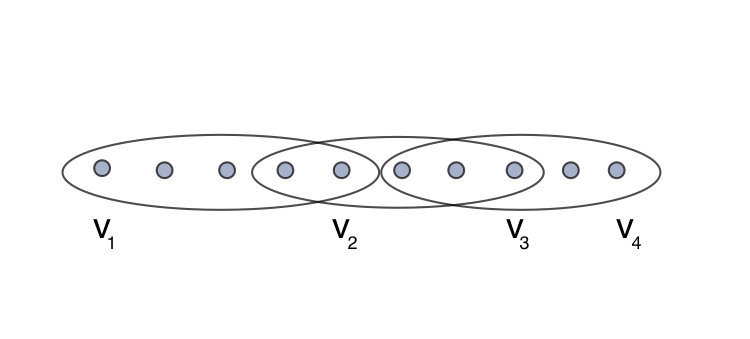}} }\caption{A $5$-uniform Berge path on distinct vertices $v_1, v_2, v_3, v_4$.} \label{BergePath}
\end{figure}
Figure \ref{BergePath} gives an example of a $5$-uniform Berge path on distinct vertices $v_1, v_2, v_3, v_4$.  It should be noted that although the four indicated vertices are distinct and the hyperedges are distinct, other vertices are allowed to be repeated in different hyperedges.

A hypergraph $H$ is \textit{connected} if there exists a Berge path between any two vertices in $H$. The \textit{connected components} of $H$ are the maximal connected subhypergraphs of $H$; we write $c(H)$ for the order of the largest connected component in $H$. 
A Berge path $v_1, v_2, \dots , v_k$ with hyperedges $e_1, e_2 , \dots , e_{k-1}$ can be extended to form a {\it Berge cycle} if we include a distinct hyperedge $e_k$ that includes both $v_1$ and $v_k$.

In the graph setting, paths are types of trees, but we will need to be more restrictive in the hypergraph setting.  Namely, Berge paths are too broad, leading us to the definition of a loose path.  An {\it $r$-uniform loose path} $P_m^{(r)}$ on $m$ vertices is a sequence of distinct vertices $v_1, v_2, \dots , v_m$ along with hyperedges $$e_i:=v_{(i-1)(r-1)+1}, v_{(i-1)(r-1)+2}, \dots , v_{(i-1)(r-1)+r},$$ where $i=1, 2, \dots , k$.  It necessarily follows that $m=r+(k-1)(r-1)$. Notice that consecutive hyperedges of a loose path intersect in exactly one vertex and each loose path is necessarily a Berge path.  For graphs, the definitions of loose paths and Berge paths coincide.

It is well-known that there exist several equivalent definitions of a tree in the graph context. With regard to the definitions above, Theorem \ref{equiv} provides four equivalent definitions of an {\it $r$-uniform tree}.  Note that some authors have referred to such trees as hypertrees (eg., see \cite{L}).

\begin{theorem}  The following definitions of an $r$-uniform tree $T$ are equivalent:
\begin{enumerate}
\item $T$ is an $r$-uniform hypergraph that can be formed hyperedge-by-hyperedge with each new hyperedge intersecting the previous hypergraph at exactly one vertex.  That is, each new hyperedge requires the creation of exactly $r-1$ new vertices.
\item $T$ is a connected $r$-uniform hypergraph that does not contain any (Berge) cycles.
\item $T$ is a connected $r$-uniform hypergraph in which the removal of any hyperedge (keeping all vertices) results in a hypergraph with exactly $r$ connected components.
\item $T$ is an $r$-uniform hypergraph in which there exists a unique loose path between any pair of distinct vertices.
\end{enumerate}\label{equiv}
\end{theorem}

\begin{proof} We prove a cyclic sequence of implications to obtain the desired theorem.
\begin{enumerate}
\item[$(1)\Rightarrow (2)$:] Suppose that $T$ is an $r$-uniform hypergraph that can be formed hyperedge-by-hyperedge with each new hyperedge intersecting the previous hypergraph at exactly one vertex. Clearly, this hypergraph and all of its subhypergraphs will have the property that any two distinct hyperedges intersect in at most one vertex. It follows that all paths in this graph will be loose paths. If there were a Berge cycle in the graph, then there must be a step in this construction process where a hyperedge is added to close off a loose path in the existing hypergraph. This would require the last hyperedge added in the Berge cycle to include at least two vertices from the previous hypergraph.  Thus, no Berge cycle can exist. \\
\item[$(2) \Rightarrow (3)$:] (by contrapositive) Suppose that $T$ is a connected $r$-uniform hypergraph such that there exists a hyperedge $e^* = v_1, v_2, \ldots, v_r$  whose removal results in a hypergraph with fewer than $r$ connected components. Then two vertices $v_i$ and $v_j$ must be in the same connected component in $T - e^*$. This means there is a Berge path connecting $v_i$ and $v_j$ which does not involve $e^*$, and adjoining $e^*$  to this Berge path gives a Berge cycle in $T$. \\
\item[$(3) \Rightarrow (4)$:] The proof is by strong induction on $s$, the number of hyperedges. In the base case (one hyperedge), both statements are automatically true. Now suppose that $(3)$ implies $(4)$ for any connected $r$-uniform hypergraph with $s$ hyperedges, and let $T$ be a connected $r$-uniform hypergraph with $s+1$ edges in which the removal of any hyperedge (keeping all vertices) results in a hypergraph with exactly $r$ connected components . We may choose an arbitrary hyperedge $e = v_1, v_2, \ldots, v_r$ and remove it, resulting in connected components $T_1, T_2, \ldots, T_r$, with $v_i \in T_i$. Note that each $T_i$ is a connected, $r$-uniform hypergraph with fewer than $s+1$ edges. Now suppose $v,w$ are distinct vertices in $T$. If $v,w$ are in the same connected component $T_i$, applying the induction hypothesis gives that there is a unique loose path in $T_i$ between $v$ and $w$. (The induction hypothesis is applicable from $(2) \implies (3)$ since each $T_i$ is a connected hypergraph with no cycles.) If $v \in T_i$ and $w \in T_j$, $i \neq j$, then  we can obtain a loose path from $v$ to $w$ by following the loose path from $v$ to $v_i$ in $T_i$, the hyperedge $e$, and the loose path from $v_j$ to $w$ in $T_j$. In either case, it is clear that the loose path between $v$ and $w$ is unique.   \\
\item[$(4) \Rightarrow (1)$:] For the sake of contradiction, suppose there is an $r$-uniform hypergraph $T$ with $s$ hyperedges for which (4) holds but (1) does not. Take $s$ to be the smallest such number.  Let $v$ and $w$ to be two vertices in $T$ so that the unique loose path between $v$ and $w$ is maximal, i.e. it is not a proper subhypergraph of any other loose path in $T$. Let $e_1, e_2, \ldots, e_k$ be the hyperedges of the loose path between $v$ and $w$, and denote by $u$ the lone vertex in $e_{k-1} \cap e_k$. Now consider the subhypergraph $T'$ obtained from $T$ by removing all vertices in $e_k$ except for $u$. Notice that $T'$ has a unique loose path between any two distinct vertices, and since $T'$ has fewer than $s$ hyperedges, $T'$ can be built hyperedge-by-hyperedge with each new hyperedge intesecting the previous hypergraph in only one vertex. However, we now add $e_k$ to $T'$ to obtain $T$, and we have actually built $T$ by such a process. This is a contradiction, which completes the proof.   
\end{enumerate}
Thus, we find that $(1)$, $(2)$, $(3)$, and $(4)$ are equivalent.
\end{proof}

Let $\delta (H)$ denote the minimal degree of any vertex in $H$, where the {\it degree} of a vertex is the number of hyperedges containing that vertex. Define a {\it free hyperedge } of an $r$-uniform hypergraph to be a hyperedge in which exactly $r-1$ vertices have degree $1$.  So, if we add a free hyperedge to a hypergraph, we add in $r-1$ new vertices along with the corresponding hyperedge.  
If $T$ is an $r$-uniform tree, then an {\it end} vertex of $T$ is a vertex of degree $1$ in a free hyperedge. The hypergraph trees that have just been defined possess a property analogous to a well-known
result regarding degrees of vertices and the existence of trees as subgraphs (e.g., see Lemma 2.1 in \cite{GV}).    

\begin{theorem}
Assume that $r\ge 2$ and let $T_m^{(r)}$ be any $r$-uniform tree of order $m$.  If $H$ is any $r$-uniform hypergraph of order $p$ with $$\delta (H)\ge \left( \begin{array}{c} p-1 \\ r-1\end{array}\right)-\left(\begin{array}{c} p-m \\ r-1 \end{array}\right),$$ then $H$ contains a subhypergraph isomorphic to $T_m^{(r)}$.
\end{theorem}

\begin{proof}
We proceed by induction on the number $k$, of hyperedges in $T_m^{(r)}$.  Note that $m=r+(k-1)(r-1)$.  If $k=1$, then $m=r$ and $$\delta (H)\ge  \left( \begin{array}{c} p-1 \\ r-1\end{array}\right)-\left(\begin{array}{c} p-r \\ r-1 \end{array}\right)>0.$$  Thus, there exists at least one hyperedge, forming a $T_r^{(r)}$. Now assume the theorem is true for all trees having $k$ hyperedges, let $T^{(r)}_{r+k(r-1)}$ be any tree with $k+1$ hyperedges, and suppose that $$\delta (H)\ge \left( \begin{array}{c} p-1 \\ r-1\end{array}\right)-\left(\begin{array}{c} p-(r+k(r-1)) \\ r-1 \end{array}\right).$$  Denote by $T'$ the tree with $k$ hyperedges formed by removing a free hyperedge (and all of its degree $1$ vertices) from $T^{(r)}_{r+k(r-1)}$ and assume that $x$ is the vertex in $T'$ that was incident with the removed leaf.  Then by the inductive hypothesis, there must be a subgraph isomorphic to $T'$.  The maximum number of hyperedges that contain $x$ and some other vertices from $T'$ in $H$ is $$ \left( \begin{array}{c} p-1 \\ r-1\end{array}\right)-\left(\begin{array}{c} p-(r+(k-1)(r-1)) \\ r-1 \end{array}\right),$$ so the assumed inequality implies that some other hyperedge that contains $x$ must exist.  Such a hyperedge can be added to $T'$ to form a copy of $T^{(r)}_{r+k(r-1)}$.
\end{proof}

\section{$n$-Good Hypergraphs}

In this section, we introduce the concept of $n$-good $r$-uniform hypergraphs.  As in the graph setting, the determination of whether or not a hypergraph is $n$-good depends on the value of a specific Ramsey number. 
Recall that an $n$-good graph $G$ is a connected graph of order $m$ that satisfies $$R(G, K_n )=R(T_m, K_n )=(m-1)(n-1)+1.$$  The fact that this concept is well-defined stems from the observation that the Ramsey number $R(T_m, K_n )$ is independent of the particular choice of tree $T_m$ of order $m$.  In the $r$-uniform hypergraph setting, it is not immediately clear that this independence is present.  So, when considering the concept of an $n$-good hypergraph, we focus on the fact that the Ramsey number $R(T_m,K_n)$ equals the lower bound proved by Chv\'atal and Harary \cite{CH}:
$$R(T_m, K_n)\ge (m-1)(n-1)+1.$$  

This result, and the corresponding upper bound proved by Chv\'atal \cite{C}, were generalized to the setting of $r$-uniform hypergraphs in \cite[Theorem 3]{BHR}.  There, it was shown that if $T_m^{(r)}$ is any $r$-uniform tree of order $m$, then \begin{equation}(m-1) \left( \ceil[\Big]{\frac{n}{r-1}} -1\right)+1 \le R(T_m^{(r)}, K_n^{(r)};r)\le (m-1)(n-1)+1.\label{ChvHar} \end{equation}  When $r=2$, these two bounds agree and the general lower bound (\ref{Burrineq}) proved by Burr \cite{B} provides no improvement to Chv\'atal and Harary's bound.

We offer the following improvement of the lower bound in (\ref{ChvHar}), which may be viewed as a generalization of Theorem 1 of \cite{Burr}. 

\begin{theorem}
Let $H_1$ and $H_2$ be $r$-uniform hypergraphs. If  $c(H_2) \geq t(H_1)$, then 
$$R(H_1,H_2;r) \geq (\chi_w(H_1) - 1)(c(H_2) - 1) + t(H_1).$$
\label {BurrGen}\end{theorem}
\vspace{-.4in}
\begin{proof}
Let $k = (\chi_w(H_1) -1)(c(H_2) - 1) + t(H_1)$. We will construct a red-blue coloring of $K_{k-1}^{(r)}$ which contains 
neither a red $H_1$ nor a blue $H_2$. Begin by taking $(\chi_w(H_1) - 1)$ disjoint copies of $K_{c(H_2)-1}^{(r)}$, along with a disjoint copy 
of $K_{t(H_1)-1}^{(r)}$. The hyperedges strictly contained in each of these complete subhypergraphs are colored blue, with all other hyperedges colored red. The order of
the largest connected component in any blue subhypergraph is $c(H_2)-1$, so no blue copy of $H_2$ can exist. Denote by $H_R$ the
subhypergraph spanned by the red hyperedges. Note that we can obtain a proper weak vertex coloring by using the same color on all vertices in each of the original disjoint complete subhypergraphs. If $t(H_1)=1$, then $\chi_w(H_R)=\chi_w(H_1)-1$, since in this case there is no $K_{t(H_1)-1}^{(r)}$ .  If $t(H_1) >1$, then $\chi_w(H_R)=\chi_w(H_1)$ and $t(H_R)=t(H_1)-1$, as any smallest color class in this coloring must have order $t(H_1)-1$.  In either case, it is clear that no red subhypergraph isomorphic to $H_1$ exists.  Thus, $R(H_1, H_2 ; r)\ge k.$
\end{proof}

With this theorem in place, we offer the following definition, generalizing the concept of a $G$-good graph as first defined by Burr \cite{Burr}.
Let $H_1$ and $H_2$ be finite hypergraphs. 
If $c(H_2)\ge t(H_1)$, the hypergraph $H_2$ is called {\it $H_1$-good} if $$R(H_1, H_2 ; r)=(\chi_w(H_1)-1)(c(H_2)-1)+t(H_1).$$   
The case where $H_1 = K_n^{(r)}$ is of special interest. Following the terminology used in \cite{BE}, we offer the following definition of an $n$-good hypergraph. (Note that in this case, the inequality $c(H_2) \ge t(K_n^{(r)})$ always holds when $H_2$ is nonempty.)

\begin{definition} The hypergraph $H_2$ is called {\it $n$-good} whenever $$R(K_n^{(r)}, H_2 ; r)=(\chi_w(K_n^{(r)})-1)(c(H_2)-1)+t(K_n^{(r)}).$$  
\end{definition}

Note that $\chi _w(K_n^{(r)})=\ceil[\big]{\frac{n}{r-1}}$ since at most $r-1$ vertices can receive the same color in any weak coloring.  If we let $n=q(r-1)+k$, where $q, k \in \Z$ with $0\le k <(r-1)$, then \begin{equation}t(K_n^{(r)})=\left\{ \begin{array}{ll} k & \mbox{if} \ k\ne 0 \\  r-1 & \mbox{if} \ k=0, \end{array}\right.\label{colorclass}\end{equation} resulting in the following corollary, which gives a slight improvement on the lower bounds given in (\ref{ChvHar}) when $H$ is a tree.  

\begin{corollary}
If $H$ is any connected $r$-uniform hypergraph of order $m\ge r$ and $n=q(r-1)+k$, where $q, k \in \Z$ with $0\le k <(r-1)$, then $$R(H, K_{n}^{(r)}; r) \ge \left\{ \begin{array}{ll} (m-1)\left( \ceil[\big]{\frac{n}{r-1}}-1\right)+ k & \mbox{if} \ k\ne 0 \\ \notag \\ (m-1)\left( \frac{n}{r-1}-1\right)+r-1 & \mbox{if} \ k=0.\end{array} \right.$$ \label{Burr}
\end{corollary}

Known lower bounds for certain $3$ and $4$-uniform Ramsey numbers allow us to verify that $K_4^{(3)}$, $K_5^{(3)}$, and $K_6^{(3)}$ are not $4$-good, $K_5^{(3)}$ is not $5$-good, and $K_5^{(4)}$ is not $5$-good (see Section 7.1 of \cite{Rad}).   It is not immediately clear whether or not any $n$-good hypergraphs exist, but trees seem like an appropriate place to begin our search since they were central to the definition in the graph setting.  So, we now focus our attention on the Ramsey numbers $R(T_m^{(r)}, K_n^{(r)}; r)$, with an emphasis on trying to determine whether or not a tree $T_m^{(r)}$ is $n$-good. 

The simplest tree $T_r^{(r)}$ consists of a single hyperedge.  From Corollary \ref{Burr}, we see that the trivial Ramsey number $R(T_r^{(r)}, K_n^{(r)};r)=n$ shows that $T_r^{(r)}$ is $n$-good for all $n\ge r$.  
Before focusing exclusively on $r$-uniform trees, we prove several bounds that hold for $R(H, K_n^{(r)};r)$.  Similar to the approach used in \cite{BE}, we offer the following theorem and corollary.

\begin{theorem}
For any connected $r$-uniform hypergraph $H$ of order $m\ge r$, $$R(H,K_{n-r+1}^{(r)};r)+m-1\le R(H, K_n^{(r)};r).$$\label{verygood}
\end{theorem}
\vspace{-.3in}
\begin{proof}
Let $s=R(H,K_{n-r+1}^{(r)};r)$ and consider a red/blue coloring of $K_{s-1}^{(r)}$ that lacks a red $H$ and a blue $K_{n-r+1}^{(r)}$.  Union this hypergraph with a red $K_{m-1}^{(r)}$, along with all interconnecting hyperedges colored blue.  Clearly, no red $H$ exists and the largest complete blue subhypergraph contains at most $n-1$ vertices.  Thus, $R(H, K_n^{(r)};r)\ge s+m-1$.
\end{proof}

\begin{corollary}
If a connected $r$-uniform hypergraph $H$ of order $m\ge r$ is $n$-good, then it is $(n-r+1)$-good. \label{goodreduction}
\end{corollary}

\begin{proof}
Suppose that $H$ is $n$-good.  That is, $$R(H, K_n^{(r)};r)= (m-1)\left( \ceil[\Big]{\frac{n}{r-1}}-1\right)+t(K_n^{(r)}).$$  From Theorem \ref{verygood}, \begin{align} R(H,K_{n-r+1}^{(r)};r) &\le R(H, K_n^{(r)};r)-m+1 \notag \\ &\le  (m-1)\left( \ceil[\Big]{\frac{n}{r-1}}-2\right)+t(K_n^{(r)}) \notag \\ &\le  (m-1)\left( \ceil[\Big]{\frac{n-r+1}{r-1}}-1\right)+t(K_n^{(r)}).\notag \end{align}  From (\ref{colorclass}), the value of $t(K_{n}^{(r)})$ is determined modulo $r-1$.  So, $t(K_{n}^{(r)})=t(K_{n-r+1}^{(r)})$, and it follows that $H$ is $(n-r+1)$-good.
\end{proof}

The following theorem should be compared to Theorem 3.2 of \cite{BE}. 

\begin{theorem}
Let $H'$ be a connected $r$-uniform hypergraph of order $m-r+1\ge r$ and let $H$ by the hypergraph formed by adding a free hyperedge to $H'$ ($H$ has order $m$).  Then $$R(H, K_{n}^{(r)};r)\le max\{ R(H' ,K_n^{(r)};r), R(H, K_{n-1}^{(r)};r)+m-r+1 \}.$$ \label{max}
\end{theorem}

\begin{proof}
Let $s=max\{ R(H' ,K_n^{(r)};r), R(H, K_{n-1}^{(r)};r)+m-r+1 \}$ and consider a red/blue coloring of the hyperedges in $K_s^{(r)}$.  If there exists a blue $K_n^{(r)}$, we are done, so suppose such a subhypergraph does not exist.  Then there must be a red $H'$.  Let $x$ be a vertex in $H'$ in which the addition of a free hyperedge incident with $x$ results in a subhypergraph isomorphic to $H$.  There are $$s-(m-r+1)\ge R(H, K_{n-1}^{(r)};r)$$ vertices not contained in the red $H'$.  If any hyperedge including $x$ and any $r-1$ of these remaining vertices is red, then we have a red $H$.  So, assume that all such hyperdges are blue.  The subhypergraph induced by the remaining vertices contains a red $H$ or a blue $K_{n-1}^{(r)}$.  In the latter case, including $x$ produced a blue $K_n^{(r)}$.
\end{proof}

\noindent We obtain the following corollary.

\begin{corollary}
Let $H'$ be an $r$-uniform hypergraph of order $m-r+1\ge r$ and let $H$ by the hypergraph formed by adding a free hyperedge to $H'$ ($H$ has order $m$).  If $n\ge r+1$, 
$$R(H', K_{n}^{(r)};r)\le n_1, \qquad \mbox{and} \qquad R(H, K_{n-1}^{(r)};r)\le n_2,$$ where $n_1\le n_2+m-r+1$, then $$R(H, K_{n}^{(r)};r)\le n_2+m-r+1.$$ \label{genconstruct1}
\end{corollary}


Using a similar construction to that of Theorem \ref{max}, the following theorem and its corollary will be useful in upcoming proofs.

\begin{theorem}
Let $H'$ be a connected $r$-uniform hypergraph of order $m-r+1\ge r$ and let $H$ by the hypergraph formed by 
adding a free hyperedge to $H'$ ($H$ has order $m$).  Then $$R(H, K_{n}^{(r)};r)\le max\{ R(H' ,K_n^{(r)};r)+n-1, R(H, K_{n-1}^{(r)};r)\}.$$ \label{max2}
\end{theorem}

\begin{proof}
Let $s=max\{ R(H' ,K_n^{(r)};r)+n-1, R(H, K_{n-1}^{(r)};r) \}$ and consider a red/blue coloring of the hyperedges in $K_s^{(r)}$.  If there exists a red $H$, we are done, so suppose such a subhypergraph does not exist.  Then there must be a blue $K_{n-1}^{(r)}$.  There are $$s-(n-1)\ge R(H', K_{n}^{(r)};r)$$ vertices not contained in the red $K_{n-1}^{(r)}$, so assume it contains a red $H'$.  Let $x$ be a vertex in $H'$ in which the addition of a free hyperedge incident with $x$ results in a subhypergraph isomorphic to $H$.   If any hyperedge including $x$ and any $r-1$ vertices from the blue $K_{n-1}^{(r)}$ is red, then we have a red $H$.  Otherwise, all such hyperdges are blue and we have a blue $K_n^{(r)}$.  
\end{proof}

\begin{corollary}
Let $H'$ be an $r$-uniform hypergraph of order $m-r+1\ge r$ and let $H$ by the hypergraph formed by adding a free hyperedge to $H'$ ($H$ has order $m$).  If $n\ge r+1$, 
$$R(H', K_{n}^{(r)};r)\le n_1, \qquad \mbox{and} \qquad R(H, K_{n-1}^{(r)};r)\le n_2,$$ where $n_2\le n_1+n-1$, then $$R(H, K_{n}^{(r)};r)\le n_1+n-1.$$ \label{genconstruct2}
\end{corollary}


Before we shift our focus to trees, we conclude this section with a proof that whenever an $r$-uniform hypergraph is $n$-good, a finite number of disjoint copies of that hypergraph is $n$-good.  When $a\in \N$ and $H$ is any $r$-uniform hypergraph, we denote by $aH$ the disjoint union of $a$ copies of $H$.

\begin{theorem}
If an $r$-uniform hypergraph $H$ of order $m$ is $n$-good, where $n\ge 2r-1$, then $aH$ is $n$-good for all $a\in \N$.
\end{theorem}

\begin{proof}
Assuming that $H$ is $n$-good, where $n\ge 2r-1$, it follows that $$R(H, K_n^{(r)};r)=(m-1)\left( \ceil[\bigg]{\frac{n}{r-1}}-1\right)+t(K_n^{(r)}).$$ It remains to be shown that \begin{equation} R(aH, K_n^{(r)};r)\le (am-1)\left( \ceil[\bigg]{\frac{n}{r-1}}-1\right)+t(K_n^{(r)}).\label{need}\end{equation}  By Lemma 3.1 of \cite{OR} (which generalized the analogous result for graphs from \cite{BES}), it follows that \begin{equation} R(aH, K_n^{(r)};r)\le (m-1)\left( \ceil[\bigg]{\frac{n}{r-1}}-1\right) +(a-1)m+t(K_n^{(r)}).\label{given}\end{equation}  Proving that (\ref{need}) follows from (\ref{given}) is equivalent to proving that $$2(a-1)\le \left( \ceil[\bigg]{\frac{n}{r-1}}-1\right)(a-1),$$ which is true whenever  $\ceil[\big]{\frac{n}{r-1}}\ge 3$ (equivalently, $n\ge 2r-1$).
\end{proof}

\section{Are $r$-Uniform Trees $n$-Good?}

While we will not obtain a complete answer to the question ``Are $r$-uniform trees $n$-good?'', we will provide an infinite number of cases in which the answer is ``yes'' and we will not obtain a single counterexample.  The usefulness of Corollaries \ref{genconstruct1} and \ref{genconstruct2} in determining upper bounds for tree/complete hypergraph Ramsey numbers follows from equivalent definition $(1)$ of an $r$-uniform tree (given in Theorem \ref{equiv}).  For example, we offer the following application of Corollary~\ref{genconstruct2} concerning the (unique) $r$-uniform tree of order $2r-1$.

\begin{theorem}
For all $r\ge 3$, $$2r\le R(T_{2r-1}^{(r)}, K_{r+1}^{(r)};r)\le 2r+1.$$
\end{theorem}

\begin{proof}
The lower bound follows from Theorem \ref{BurrGen}.  The upper bound is a direct application of Corollary \ref{genconstruct2} applied to the trivial Ramsey numbers $$R(T_r^{(r)}, K_{r+1}^{(r)};r)=r+1=m_1 \quad \mbox{and} \quad R(T_{2r-1}^{(r)}, K_{r}^{(r)};r)=2r-1=m_2.$$  It is easily confirmed that $m_2\le m_1+r$, as is required to apply Corollary \ref{genconstruct2}. 
\end{proof}

Let $T$ be an $r$-uniform tree containing exactly $t$ hyperedges.
In 2009, Loh \cite{L} solved a problem posed by Bohman, Frieze, and Mubayi \cite{BFM} when he proved that if an $r$-uniform hypergraph $H$ satisfies $\chi _w (H)>t$, then $T$ is isomorphic to a subhypergraph of $H$.  This result is independent of the specific tree being considered and it is independent of the uniformity $r$.  As an application of this result, Loh proved the following upper bound for $R(T_m^{(r)}, K_n^{(r)};r)$, which is an $r$-uniform analogue for the bound proved by Chv\'atal \cite{C} for graphs.  We also note that this upper bound improves on the upper bound given in Theorem 3.4 of \cite{BHR}.  We reproduce Loh's proof for completion and to provide the proof in the context of the given paper.

\begin{theorem}[Loh, 2009]
If $n\ge r\ge 2$ and $T_m^{(r)}$ is any $r$-uniform tree on $m$ vertices, then $$R(T_m^{(r)}, K_{n}^{(r)};r)\le \frac{(m-1)(n-1)}{r-1}+1.$$\label{Loh}
\end{theorem}

\begin{proof}
Let $p=\frac{(m-1)(n-1)}{r-1}+1$ and suppose that $T_m^{(r)}$ contains exactly $t$ hyperedges.  Then $t=\frac{m-1}{r-1}$ and we consider a red/blue coloring of the hyperedges of $K_p^{(r)}$.  Denote the subhypergraphs spanned by the red and blue hyperedges by $H_R$ and $H_B$, respectively.  If $\chi _w (H_R) \le t$, then any collection of vertices of the same color form an independent set and it follows that $H_R$ has an independent set of cardinality at least $$\ceil[\Big]{\frac{p}{t}}=\ceil[\bigg]{\frac{(n-1)t+1}{t}}=n.$$  Such a collection of vertices corresponds to a subhypergraph of $H_B$ that is isomorphic to $K_n^{(r)}$.  If $\chi _w(H_R)>t$, then Theorem 1 of \cite{L} implies that $H_R$ contains a subhypergraph isomorphic to $T_m^{(r)}$.  It follows that every red/blue coloring of the hyperedges of $K_p^{(r)}$ contains a red $T_m^{(r)}$ or a blue $K_n^{(r)}$.
\end{proof}

\begin{corollary}
Let  $n\ge r\ge 2$ and $T_m^{(r)}$ be any $r$-uniform tree.  If $r-1$ divides $n-1$, then $T_m^{(r)}$ is $n$-good.  \label{Loh2}
\end{corollary}

\begin{proof}
Writing $(r-1)\ell =n-1$, it follows that $t(K_n^{(r)})=1$ and $\ceil[\big]{\frac{n}{r-1}}=\ell +1$.  Thus, the lower bound in Theorem \ref{BurrGen} becomes $(m-1)(\ell)+1$, agreeing with the upper bound in Theorem \ref{Loh}.
\end{proof}

\noindent In particular, note that when $r-1$ divides $n-1$, the Ramsey number $R(T_m^{(r)}, K_n^{(r)};r)$ is independent of the choice of $r$-uniform tree on $m$ vertices.

Next, we consider a construction that provides an upper bound for all Ramsey numbers $R(T_{2r-1}^{(r)}, K_n^{(r)};r)$ when $r\ge 3$ is odd.  In the next section, this upper bound will prove to be tight in the case of $r=3$.  Before we state and prove this theorem, let us recall a few definitions. If $H$ is a hypergraph, a {\em matching on H} is a set of hyperedges from $H$ which are all disjoint from one another. The {\em size} of a matching $M$ is the number of hyperedges in $M$. A matching is {\em maximal} if it is not a proper subset of any other matching.

\begin{theorem}
Let $r\ge3$ be odd and $n\ge r+1$.  Then $R(T_{2r-1}^{(r)}, K_n^{(r)};r)\le p,$ where $$p=\left\{ \begin{array}{lr} \frac{r+1}{2}n-(r-1) & \mbox{if}\ n \ \mbox{is even} \\ \\ \frac{r+1}{2}n-\frac{r-1}{2} & \mbox{if}\ n \ \mbox{is odd.} \end{array} \right.$$ \label{tree2}
\end{theorem}

\begin{proof}
Let $p$ be defined as above and consider a red/blue coloring of the hyperedges in $K_p^{(r)}$ that lacks a red subhypergraph isomorphic to $T_{2r-1}^{(r)}$.  Suppose that $M$ is a maximal red matching of size $k$ and define $S_1$ to be a set of vertices consisting of exactly two vertices from each hyperedge in $M$.  Then the subhypergraph induced by $S_1$ is a complete blue hypergraph as the assumption that $r$ is odd forces every hyperedge to intersect some element of $M$ at exactly one vertex.  If $2k\ge n$, then the coloring contains a blue $K_n^{(r)}$.  Otherwise, $2k\le n-1$ and we can define $S_2$ to consist of all vertices not in $M$ along with a single vertex from each hyperedge in $M$.  In the case where $n$ is even, equality isn't possible, so this inequality can be improved to $2k\le n-2$.  When $n$ is odd, we have that \begin{align} |S_2|&=p-(r-1)k \notag \\ &=\frac{r+1}{2}n-\frac{r-1}{2}-(r-1)k \notag \\ &\ge \frac{r+1}{2}n-\frac{r-1}{2}+(r-1)\left(\frac{1-n}{2}\right) =n.\notag \end{align} Similarly, when $n$ is even, we have \begin{align} |S_2|&=p-(r-1)k \notag \\ &=\frac{r+1}{2}n-(r-1)-(r-1)k \notag \\ &\ge \frac{r+1}{2}n-(r-1)+(r-1)\left(\frac{2-n}{2}\right) =n.\notag \end{align}  In both cases, the subhypergraph induced by $S_2$ forms a complete blue hypergraph of order at least $n$, completing the proof of the theorem.
\end{proof}

Now, we restrict our attention in the next section to showing that certain $3$-uniform trees are $n$-good and to proving upper bounds whenever our methods are insufficient for determining exact evaluations of $R(T_m^{(3)}, K_n^{(3)};3)$.

\section{$3$-Uniform $n$-Good Hypergraphs}

Having laid down the appropriate framework with which to study $n$-good hypergraphs when $r\ge 3$, we now focus on the $3$-uniform case.  The smaller uniformity will enable us to give numerous precise evaluations of $R(H,K_n;3)$, from which we can gain a better understanding of which hypergraphs are $n$-good.

\subsection{Trees Versus Complete $3$-Uniform Hypergraphs}\label{3uniform}

First, we focus on finding exact Ramsey numbers for certain trees and complete graphs in the 3-uniform case.  The first nontrivial tree to consider is $T_5^{(3)}$, which is a loose path with two hyperedges and is unique up to isomorphism.  It is easily confirmed that the lower bound given in Theorem \ref{BurrGen} and the upper bound given in Theorem \ref{tree2} agree in this case, implying 

\begin{equation} R(T_{5}^{(3)}; K_{n}^{(3)}; 3) = \left\{ \begin{array}{rl} 2n-2 & \mbox{if $n$ is even} \\ 2n-1 & \mbox{if $n$ is odd.} \end{array} \right.\label{5trees}\end{equation}
Of course, when $n$ is odd, this result also follows from Theorem \ref{Loh} (and Corollary \ref{Loh2}).
Therefore, $T_5^{(3)}$ is $n$-good for all $n\ge 3$.  

Our efforts in this section lead to the determination of the values/ranges for $R(T_m^{(3)}, K_n^{(3)};3)$ given in Table \ref{t1}.  All exact evaluations included in this table correspond to trees that are $n$-good.

\begin{table}[H]
\centering
\begin{tabular}{|c||c|c|c|c|c|c|c|}
	\hline
	$_m\ \backslash \ ^n $  & 4 & 5 & 6 & 7 & 8 & 9 & 10   \\
	\hline\hline
	5  & 6 & 9 & 10 & 13 & 14 & 17 & 18   \\
	\hline
	7  & [8, 9] & 13 & [14, 15] & 19 & [20, 21] & 25 & [26, 27]   \\
	\hline
	9  & [10, 12] & 17 & [18, 20] & 25 & [26, 28] & 33 & [34, 36]    \\
	\hline
	11  & [12, 15] & 21 & [22, 25] & 31 & [32, 35] & 41 & [42, 45]    \\
	\hline 
	13  & [14, 18] & 25 & [26, 30] & 37 & [38, 42] & 49 & [50, 54]    \\
	\hline 
	15  & [16, 21] & 29 & [30, 35] & 43 & [44, 49] & 57 & [58, 63]    \\
	\hline \hline
	2j+1  & [2j+2, 3j] & 4j+1 & [4j+2, 5j] & 6j+1 & [6j+2, 7j] & 8j+1 & [8j+2, 9j]  \\
	\hline 
\end{tabular}
\caption{Exact values/ranges for $R(T_m^{(3)}, K_n^{(3)};3)$ whenever $m=2j+1\ge 5$ and $4\le n\le 10$.  All of the exact values shown in this chart correspond to trees that are $n$-good.}
\label{t1}
\end{table}

\begin{theorem}
For all $n\ge 3$ and $j\ge 2$, $$R(T_{2j+1}^{(3)}, K_{n}^{(3)};3)=j(n-1)+1$$ when $n$ is odd, and $$j(n-2)+2\le R(T_{2j+1}^{(3)}, K_{n}^{(3)};3)\le j(n-1)$$ when $n$ is even.
\label{treebounds}\end{theorem}

\begin{proof}
All lower bounds follow from Theorem \ref{BurrGen} and the upper bound in the case when $n$ is odd follows from Theorem \ref{Loh} (and Corollary \ref{Loh2}).  To prove the upper bounds for a fixed even value of $n$, we proceed by induction on $j$, using $R(T_5^{(3)}, K_n^{(3)};3)$ as the base case.  Suppose that the upper bound holds for all $3$-uniform trees having fewer than $j$ hyperedges and let $T_{2j+1}^{(3)}$ be an $r$-uniform tree containing exactly $j$ hyperedges.  Let $T'$ be the tree formed by removing a single free hyperedge (along with all $r-1$ of its degree $1$ vertices) from $T_{2j+1}^{(3)}$.  By the inductive hypothesis, we have that $$R(T', K_n^{(3)};3)\le (j-1)(n-1),$$ and since $n-1$ is odd, we have that $$R(T_{2j+1}{(3)},K_{n-1}^{(3)};3)=j(n-2)+1.$$  Letting $$n_1=(j-1)(n-1) \quad \mbox{and} \quad n_2=j(n-2)+1,$$ it is easily confirmed that $n_2\le  n_1+n-1$.  Thus, from Corollary \ref{genconstruct2}, it follows that $$R(T_{2j+1}^{(3)}, K_n^{(3)};3)\le jn-j=j(n-1)$$ when $n$ is even.
\end{proof}

Although Theorem \ref{treebounds} is the best we can offer for arbitrary $3$-uniform trees, stronger upper bounds can be determined for the special case of loose paths when $n$ is even and this is the the focus of Subsection \ref{loose}.  

\subsection{Loose Paths Versus Complete $3$-Uniform Hypergraphs}\label{loose}

Now, we turn our attention to improving the bounds of $R(T_m^{(3)}, K_n^{(3)};3)$ when $T_m^{(3)}$ is the loose path $P_m^{(3)}$.  In the following theorem, we will show that the $3$-uniform loose path $P_m^{(3)}$ is $4$-good  (where $m$ is odd).

\begin{theorem}
If $j\ge 1$, then
$R(P_{2j+1}^{(3)},K_4^{(3)};3)=2j+2$. \label{4good}
\end{theorem}

\begin{proof}
Theorem \ref{BurrGen} shows that $2j+2$ is a lower bound for the given Ramsey number for all values of $j$, so now it remains to be shown that $R(P_{2j+1}^{(3)},K_4^{(3)};3)\le 2j+2$.  We proceed by (weak) induction on $j$.  The $j=1$ case follows from the trivial Ramsey number $R(P_{3}^{(3)}, K_4^{(3)};3)=4$.  Now, assume that $R(P_{2j-1}^{(3)},K_4^{(3)};3)\le 2(j-1)+2$  and consider a red/blue coloring of $K_{2j+2}^{(3)}$.  By the inductive hypothesis, there exists a red $P_{2j-1}^{(3)}$ or a blue $K_4^{(3)}$.  In the latter case, we are done, so assume the former case and let $x_1$, $x_2$, $x_3$, and $x_4$ be the end vertices in the red $P_{2j-1}^{(3)}$.  Other than the red $P_{2j-1}^{(3)}$, there exist three other vertices; label them $y_1$, $y_2$, and $y_3$.  There are now two cases to consider: the hyperedge $y_1y_2y_3$ is either red or it is blue.  

\begin{enumerate}
\item[Case 1:] Assume $y_1y_2y_3$ is blue.  Then if any of the hyperedges $x_1y_1y_2$, $x_1y_2y_3$, or $x_1y_1y_3$ is red, the path extends to form a red $P_{2j+1}^{(3)}$.  Otherwise, all of these hyperedges are blue and the subhypergraph induced by $\{ x_1, y_1, y_2, y_3\}$ is a blue $K_4^{(3)}$. 
\item[Case 2:] Assume $y_1y_2y_3$ is red and label the hyperedges in the red $P_{2j-1}^{(3)}$ by $e_1, e_2, \dots , e_{j-1}$, where $e_i$ is adjacent to $e_{i+1}$ for each $1\le i\le j-2$.  Without loss of generality, assume that the end vertex $x_1$ is contained in $e_1$ and $x_2$ is contained in $e_{j-1}$.  Now consider the subhypergraph induced by $\{ x_1, x_2, y_1, y_2\}$.  It either forms a blue $K_4^{(3)}$ or some hyperedge is red.  If either $x_\ell y_1y_2$ is red for $\ell =1, 2$, then it is clear that we can form a red $P_{2j+1}^{(3)}$ by adding this hyperedge to the corresponding end of the red $P_{2j-1}^{(3)}$.  If $x_1x_2y_\ell$ is red for $\ell =1, 2$, then the hyperedges $$y_1y_2y_3 - x_1x_2y_\ell - e_1 - e_2 - \cdots - e_{j-2}$$ form a red $P_{2j+1}^{(3)}$.  
\end{enumerate}
Hence, regardless of how we color the hyperedges in $K_{2j+2}^{(3)}$, we are able to prove the existence of a red $P_{2j+1}^{(3)}$ or a blue $K_4^{(3)}$.
\end{proof}
\begin{theorem}
The loose path $P_7^{(3)}$ satisfies $$R(P_7^{(3)}, K_8^{(3)};3)=20.$$
\label{P78good}
\end{theorem}

\begin{proof}
It suffices to prove that $R(P_7^{(3)}, K_8^{(3)}; 3)\le 20$, so consider a red/blue coloring of the hyperedges in $K_{20}^{(3)}$.  Since $R(P_5^{(3)}, K_{8}^{(3)};3)=14$ (equation (\ref{5trees})), there exists a red $P_5^{(3)}$ or a blue $K_8^{(3)}$.  Assume the former case and consider the subhypergraph induced on the 15 vertices not included in the red $P_5^{(3)}$. We again apply the same Ramsey number to see that there is a red $P_5^{(3)}$ or a blue $K_8^{(3)}$.  Assume the former case so that we now have two disjoint red subhypergraphs isomorphic to $P_5^{(3)}$, along with ten other vertices.  Denote the paths by $P_1$ and $P_2$ and let their hyperedges be given by $$w_1w_2w_3 - w_3w_4w_5 \qquad \mbox{and} \qquad x_1x_2x_3 - x_3x_4x_5,$$ respectively.   Since $R(P_5^{(3)}, K_6^{(3)};3)=10$ (equation (\ref{5trees})), it follows that the subhypergraph induced on the remaining ten vertices contains a red $P_5^{(3)}$ or a blue $K_6^{(3)}$, giving us two cases to consider.
\begin{enumerate}
\item[Case 1:] If there exists a blue $K_6^{(3)}$, then it is disjoint from the two paths (see Figure \ref{fig1}).  
\begin{figure}[H]
\centerline{{\includegraphics[width=0.8\textwidth]{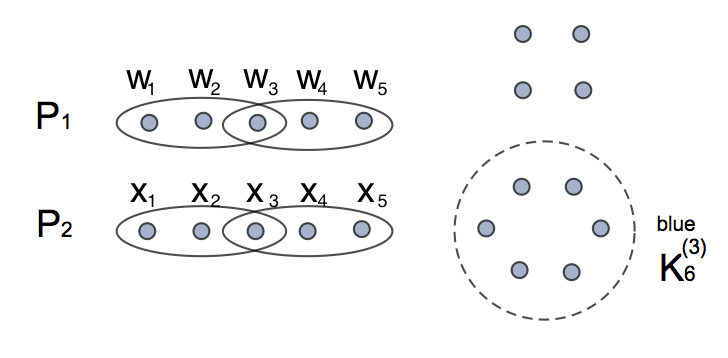}} }\caption{A coloring of $K_{20}^{(3)}$ that contains two disjoint red $P_5^{(3)}$ subhypergraphs and a disjoint blue $K_6^{(3)}$.} \label{fig1}
\end{figure}
\noindent If we denote the vertices in the blue $K_6^{(3)}$ by $z_1, z_2, \dots , z_6$, then consider the hyperedges of the forms $w_5x_5 z_i$, $w_5 z_iz_j$, and $x_5 z_iz_j$, where $i\ne j$.  If any such hyperedge is red, we can form a red $P_7^{(3)}$.  Otherwise, all such hyperedges are blue and the subhypergraph induced by $\{w_5, x_5, z_1, z_2, \dots , z_6 \}$ is a blue $K_8^{(3)}$.
\item[Case 2:] If there exists a red $P_5^{(3)}$, call it $P_3$ and denote its hyperedges by $y_1y_2y_3 - y_3y_4y_5$  (See figure \ref{fig2}).  
\begin{figure}[H]
\centerline{{\includegraphics[width=0.8\textwidth]{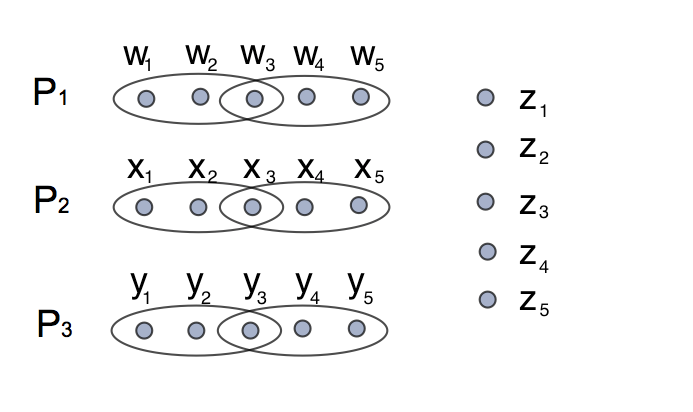}} }\caption{A coloring of $K_{20}^{(3)}$ that contains three disjoint red $P_5^{(3)}$ subhypergraphs.} \label{fig2}
\end{figure}
\noindent Denote the vertices not contained in $P_1$, $P_2$, or $P_3$ by $z_1, z_2, \dots , z_5$.  The subhypergraph induced by $\{ z_1, z_2, \dots , z_5\}$ either contains a red hyperedge or it does not, giving us two subcases to consider.
\begin{enumerate}
\item[Subcase 1:] Suppose that the subhypergraph induced by $\{ z_1, z_2, \dots , z_5\}$ is a blue $K_5^{(3)}$.  If any of the hyperedges $w_5x_5y_5$, $w_5x_5z_i$, $w_5y_5z_i$, $x_5y_5z_i$, $w_5z_iz_j$, $x_5z_iz_j$, or $y_5z_iz_j$ are red, where $i\ne j$, then we can form a red $P_7^{(3)}$.  If they are all blue, then the subhypergraph induced by $$\{w_5, x_5, y_5, z_1, z_2, \dots , z_5 \}$$ is a blue $K_8^{(3)}$.
\item[Subcase 2:]  If the subhypergraph induced by $\{ z_1, z_2, \dots , z_5\}$ contains a red hyperedge, then without loss of generality, suppose that $z_1z_2z_3$ is red.  Consider the subhypergraph that is induced by $\{ w_1, w_5, x_1, x_5, y_1, y_5, z_1, z_2\}$.  If any hyperedge is red, we can extend one of the paths to form a red $P_7^{(3)}$.  If they are all blue, we have a blue $K_8^{(3)}$.
\end{enumerate}
\end{enumerate}
Thus, in all cases, we have proven that our coloring of $K_{20}^{(3)}$ contains a red $P_7^{(3)}$ or a blue $K_8^{(3)}$.
\end{proof}

\begin{corollary}
For $j\ge 3$, $$R(P_{2j+1}^{(3)}, K_8^{(3)};3)\le 7j-1.$$
\end{corollary}

\begin{proof}
We proceed by (weak) induction on $j\ge 3$.  Theorem \ref{P78good} provides the base case ($j=3$).  Now suppose that $$R(P_{2j-1}^{(3)}, K_8^{(3)};3)\le 7(j-1)-1=7j-8.$$  Using the fact that $$R(P_{2j+1}^{(3)}, K_7^{(3)};3)=6j+1,$$ one can check that the criteria for applying Theorem \ref{genconstruct2} are met and $$R(P_{2j+1}^{(3))}, K_8^{(3)};3)\le 7j-1,$$ proving the corollary.
\end{proof}

\begin{corollary}
The Ramsey number $$R(P_7^{(3)}, K_{6}^{(3)};3)=14.$$\label{P76good}
\end{corollary}

\begin{proof}
This corollary follows from Theorem \ref{P78good} and Corollary \ref{goodreduction}:  if $P_7^{(3)}$ is  $8$-good, then it is $6$-good. 
\end{proof}

\begin{corollary}
For $j\ge 3$, $$R(P_{2j+1}^{(3)}, K_6^{(3)};3)\le 5j-1.$$
\end{corollary}

\begin{proof}
We proceed by (weak) induction on $j\ge 3$.  Corollary \ref{P76good} provides the base case ($j=3$).  Now suppose that $$R(P_{2j-1}^{(3)}, K_6^{(3)};3)\le 5(j-1)-1=5j-6.$$  Using the fact that $$R(P_{2j+1}^{(3)}, K_5^{(3)};3)=4j+1,$$ one can check that the criteria for applying Lemma \ref{genconstruct2} are met and $$R(P_{2j+1}^{(3))}, K_6^{(3)};3)\le 5j-1,$$ proving the corollary.
\end{proof}

The previous two theorems (and three corollaries) allow us to improve several of the known upper bounds in Table \ref{t1} when the tree being considered is a loose path.  So, we provide the following table of exact values/ranges for Ramsey numbers of the form $R(P_m^{(3)}, K_n^{(3)};3)$.  As before, exact evaluations correspond to loose paths that are $n$-good.

\begin{table}[!h]
\centering
\begin{tabular}{|c||c|c|c|c|c|c|c|}
	\hline
	$_m\ \backslash \ ^n $  & 4 & 5 & 6 & 7 & 8 & 9 & 10   \\
	\hline\hline
	5  & 6 & 9 & 10 & 13 & 14 & 17 & 18   \\
	\hline
	7  & 8 & 13 & 14 & 19 & 20 & 25 & [26, 27]   \\
	\hline
	9  & 10 & 17 & [18, 19] & 25 & [26, 27] & 33 & [34, 36]    \\
	\hline
	11  & 12 & 21 & [22, 24] & 31 & [32, 34] & 41 & [42, 45]    \\
	\hline 
	13  & 14 & 25 & [26, 29] & 37 & [38, 41] & 49 & [50, 54]    \\
	\hline 
	15  & 16 & 29 & [30, 34] & 43 & [44, 48] & 57 & [58, 63]    \\
	\hline \hline
	2j+1  & 2j+2 & 4j+1 & [4j+2, 5j-1] & 6j+1 & [6j+2, 7j-1] & 8j+1 & [8j+2, 9j]  \\
	\hline
\end{tabular}
\caption{Exact values/ranges for $R(P_m^{(3)}, K_n^{(3)};3)$ whenever $m=2j+1\ge 5$ and $4\le n\le 10$.  All of the exact values shown in this chart correspond to loose paths that are $n$-good.}
\label{t2}
\end{table}

\subsection{Cycles Versus Complete $3$-Uniform Hypergraphs}

Having identified an infinite number of $n$-good trees without encountering a counterexample, we conclude this section with some nontrivial examples of hypergraphs that are not $n$-good.  Define the loose cycle $C_4^{(3)}$ to consist of two $3$-uniform hyperedges whose intersection has two elements.

\begin{theorem}
The 3-uniform hypergraph $C_4^{(3)}$ is $4$-good. 
\end{theorem}

\begin{proof}
Consider a red-blue coloring of the hyperedges in $K_{5}^{(3)}$. If there are three hyperedges $A_1, A_2, A_3$, then by the Inclusion-Exclusion Principle,
\begin{align*}
5 &\geq |A_1 \cup A_2 \cup A_3| \\
&\geq |A_1| + |A_2| + |A_3| - |A_1 \cap A_2| - |A_1 \cap A_3| - |A_2 \cap A_3| 
\end{align*}
If all pairs of subsets have intersections or size one or less, then we would have $5 \geq 6$. Therefore if there are at least three red hyperedges, there is a red copy of $C_4^{(3)}$. 
Now suppose there are exactly two red hyperedges $A_1$ and $A_2$. If $|A_1 \cap A_2| = 2$, then $A_1$ and $A_2$ form a red copy of $H$. Otherwise, we must have $|A_1 \cap A_2| = 1$, say $A_1 = \{a_1, a_2, a_3\}$ and $A_2 = \{a_3, a_4, a_5\}$, with $a_i \neq a_j$ for $i \neq j$. In this case, all hyperedges not containing $a_3$ must be blue, so there is a blue complete hypergraph on the four vertices $a_1, a_2, a_4$ and $a_5$. 
In either case, there is a either a red copy of $C_4^{(3)}$ or a blue copy of $K_4^{(3)}$.
\end{proof}

The proof of the following Theorem will be simplified by carefully stating the specific conditions under which the cycle $C_4^{(3)}$ is $n$-good.  Let $H$ be a 3-uniform hypergraph of order $m$. If $n$ is even, then $H$ is $n$-good if and only if $$R(H,K_n^{(3)};3)\le (m-1)(n/2-1)+2.$$  If $n$ is odd, then $H$ is $n$-good if and only if $$R(H,K_n^{(3)};3)\le (m-1)((n+1)/2-1)+1.$$ For $C_{4}^{(3)}$, these bounds are given explicitly by: if $n$ is even, then $C_{4}^{(3)}$ is $n$-good if and only if $$R(C_{4}^{(3)}, K_{n}^{(3)}; 3) \le (4-1)(n/2 -1)+2 = 3n/2 +1.$$ If $n$ is odd, then $C_4^{(3)}$ is $n$-good if and only if $$R(C_4^{(3)},K_n^{(3)};3)\le 3((n+1)/2-1)+1 = 3/2(n+1) - 2.$$

\begin{theorem}
If $j$ is an even positive integer such that $3j+1$ is prime, then $R(C_{4}^{(3)}, K_{2j+1}^{(3)};3) > 3j+1$.
\end{theorem}

\begin{proof}
Let $j$ be an even positive integer such that $3j+1$ is prime, and let $p = 3j+1$. We identify the vertices of $K_{p}^{(3)}$ with the elements $\{0,1,2,\ldots,p-1\}$ of the finite field  $\mathbb{Z}/{p\mathbb{Z}}$, and denote by $\left( \mathbb{Z}/{p\mathbb{Z}} \right)^*$ the multiplicative group of nonzero elements of $\mathbb{Z}/{p\mathbb{Z}}$.  We take two types of hyperedges to be red. The first type consists of 3-element sets of the form $\{x, -x, 0 \}$, while the second type consists of $3$-element subsets $\{x,y,z\}$ such that $x^3 = y^3 = z^3$ i.e. the $j$ cosets of the kernel of the group homomorphism $a \mapsto a^3$. As such, it is immediate that two red hyperedges of the second kind will have disjoint intersection. It is also obvious that two red hyperedges of the first kind can only intersect in the singleton set $\{0\}$. Also, a red hyperedge of the first and second type can intersect in only a subset of size 1, since $(-x)^3 = -x^3 \neq x^3$ for any nonzero $x$ in $\mathbb{Z}/{p\mathbb{Z}}$. Thus there is no red $C_{4}^{(3)}$.

Now suppose that $K_m$ is a complete blue subhypergraph. If $0$ is a vertex of $K_m$, then $K_m$ can include at most one element from each of the $\frac{p-1}{2}$ pairs $\{x,-x\}$, and hence  \[
m \leq \frac{p-1}{2} + 1 = \frac{3j}{2} + 1 < 2j + 1
\]
If $0$ is not a vertex in $K_m$, then $K_m$ can include at most two elements from each of the $j$ cosets, so $m \leq 2j < 2j+1$. As we have exhibited a red-blue coloring of $K_{3j+1}^{(3)}$ which contains neither a red $C_{4}^{(3)}$ nor a blue $K_{2j+1}$, this shows that $R(C_{4}^{(3)}, K_{2j+1}^{(3)};3) > 3j+1$.
 \end{proof}
 
 \begin{corollary}
 If $j$ is an odd positive integer such that $3j+1$ is prime, then $C_{4}^{(3)}$ is not $(2j+1)$-good.
 \end{corollary}
 
 \begin{proof}
From the preceding Theorem, $R(C_{4}^{(3)},K_{2j+1};3) > 3j+1 = \frac{3}{2}(2j+2) - 2.$
 \end{proof}

\begin{corollary}
$C_4^{(3)}$ is not $5$-good.
\end{corollary}

\begin{proof}
This is an application of the preceding Corollary. 
\end{proof}

\noindent We leave open the general question of determining which other $3$-uniform cycles are $n$-good.

\section{Conclusion}

From this work, we know that there are infinitely many 3-uniform trees that are $n$-good. We have also found additional values for which the $3$-uniform path $P^{(3)}_m$ is $n$-good, as well as determining when certain cycles are \textit{not} $n$-good. We conclude by stating a few conjectures that follow from our work.  

\begin{conjecture}
If $r\ge 2$ and $T_1$ and $T_2$ are any $r$-uniform trees of order $m$, then $$R(T_1, K_n^{(r)};r)=R(T_2, K_n^{(r)};r).$$
\end{conjecture}

\noindent A stronger statement is contained in the following conjecture.

\begin{conjecture}
If $r\ge 2$ and $T$ is any $r$-uniform tree, then $T$ is $n$-good. 
\end{conjecture}
 
\noindent We stated the above conjectures for general $r$-uniform hypergraphs, but even proving them for the $3$-uniform case would be substantial.  Other directions for future inquiry may also include exploring properties of $H_1$-good hypergraphs for cases in which $H_1$ is not complete.

\bibliographystyle{amsplain}

\end{document}